\documentclass[11pt]{amsart}
\usepackage[utf8]{inputenc}
\usepackage{amsmath}
\usepackage{amssymb}
\usepackage{amsthm}
\usepackage{color}
\usepackage{bbm}

\newtheorem{lem}{Lemma}
\newtheorem{prop}{Proposition}
\newtheorem{thm}{Theorem}

\theoremstyle{definition}
\newtheorem{dfn}{Definition}

\newcommand{\bbk}{\mathbbm{k}}

\newcommand{\Z}{\mathbb{Z}}
\newcommand{\Q}{\mathbb{Q}}
\newcommand{\R}{\mathbb{R}}

\newcommand{\di}[2]{\begin{bmatrix}#1\\#2\end{bmatrix}}
\allowdisplaybreaks
\numberwithin{equation}{section}

\title[Consistency relations of rank 2 cluster scattering diagrams]
{Consistency relations of rank 2 cluster scattering diagrams of affine type and
pentagon relation}
\author{Kodai Matsushita}
\address{(Kodai Matsushita) Graduate School of Mathematics, Nagoya University, Chikusa-ku, Nagoya, 464-8602 Japan}
\email{m19043g@math.nagoya-u.ac.jp}
\date{}
\begin{document}

\maketitle

\begin{abstract}
    In this paper, we prove the 
consistency relations of rank $2$ cluster 
scattering diagrams of affine type by using the pentagon relation. 

\end{abstract}

\section{Introduction}
The cluster algebras are commutative algebras introduced by Fomin and Zelevinsky~\cite{FZ02}. 
The positivity of coefficients of $F$-polynomials and 
the sign-coherence of $c$-vectors were important conjectures. 
These were proved in \cite{gross2018canonical} by using scattering diagram methods. 
Scattering diagrams for cluster algebras are characterized by 
the consistency relations in their structure groups $G$. 
\par
The structure group $G$ of a given cluster scattering diagram is generated by the dilogarithm elements~\cite{gross2018canonical,nakanishiIII} 
\begin{align}
    \Psi[n]:=\exp\left(
        \sum_{j=1}^{\infty}\frac{(-1)^{j+1}}{j^2}X_{jn}.
    \right). 
\end{align}
The precise definition of $G$ is given in \S 2. 
These elements satisfy pentagon relations:
\begin{align}
    \Psi[n]^c\Psi[n']^c=\Psi[n]^c\Psi[n+n']^c\Psi[n]^c
\end{align}
where $\{n,n'\}=c^{-1}$. 
\par
In this paper, we prove the 
consistency relations of rank $2$ cluster 
scattering diagrams of affine type, namely 
types $A_1^{(1)}$ and $A_2^{(2)}$. 
More precisely, 
we prove the following theorem. For simplicity, let $\di{n_1}{n_2} = \Psi[(n_1,n_2)]$. 
\begin{thm}
The following relations holds:
\begin{align}
    \begin{bmatrix}
    0\\1
    \end{bmatrix}^2
    \begin{bmatrix}
    1\\0
    \end{bmatrix}^2
    =
    \begin{bmatrix}
    1\\0
    \end{bmatrix}^2
    \begin{bmatrix}
    2\\1
    \end{bmatrix}^2
    \begin{bmatrix}
    3\\2
    \end{bmatrix}^2
    \cdots\prod_{j=0}^\infty
    \begin{bmatrix}
    2^j\\2^j
    \end{bmatrix}^{2^{2-j}}
    \cdots
    \begin{bmatrix}
    2\\3
    \end{bmatrix}^2
    \begin{bmatrix}
    1\\2
    \end{bmatrix}^2
    \begin{bmatrix}
    0\\1
    \end{bmatrix}^2, \label{A11}
\end{align}
\begin{align}\label{A22}
    \begin{bmatrix}
    0\\1
    \end{bmatrix}^4
    \begin{bmatrix}
    1\\0
    \end{bmatrix}
    &=
    \begin{bmatrix}
    1\\0
    \end{bmatrix}
    \begin{bmatrix}
    1\\1
    \end{bmatrix}^4
    \begin{bmatrix}
    3\\4
    \end{bmatrix}
    \begin{bmatrix}
    2\\3
    \end{bmatrix}^4
    \begin{bmatrix}
    5\\6
    \end{bmatrix}
    \begin{bmatrix}
    3\\5
    \end{bmatrix}^4
    \cdots\\&\quad\times
    \begin{bmatrix}
    1\\2
    \end{bmatrix}^6
    \prod_{j=1}^\infty
    \begin{bmatrix}
    2^j\\2^{j+1}
    \end{bmatrix}^{2^{2-j}}
    \cdots
    \begin{bmatrix}
    5\\12
    \end{bmatrix}
    \begin{bmatrix}
    2\\5
    \end{bmatrix}^4
    \begin{bmatrix}
    3\\8
    \end{bmatrix}
    \begin{bmatrix}
    1\\3
    \end{bmatrix}^4
    \begin{bmatrix}
    1\\4
    \end{bmatrix}
    \begin{bmatrix}
    0\\1
    \end{bmatrix}^4. 
    \notag
\end{align}
Moreover, these formulas can be reduced to
trivial relations by iteration of 
the pentagon relations.
\end{thm}
The relations (\ref{A11}) and (\ref{A22}) 
are the (unique) consistency relations of type $A_1^{(1)}$ 
and type $A_2^{(2)}$, respectively. 
\par
We say a product of dilogarthm elements is ordered,
(resp. anti-ordered)
if,
for any adjacent pair $\di{n_1}{n_2}\di{n'_1}{n'_2}$, 
the inequality 
$n_1 / n_2 \le n'_1 / n'_2$ (resp. $n_1 / n_2 \le n'_1 / n'_2$)
holds. 
The consistency relations of scattering diagrams in $\R^2$ have 
the form of 
\begin{align}
    \textrm{``anti-ordered product''} = \textrm{``ordered product''}. 
\end{align}
It was shown that the consistency relations are generated by 
the pentagon relation~\cite{nakanishiIII}, 
and the above theorem provides an simplest examples
involving the \emph{infinite} product.
\par
We remark that the relations 
(\ref{A11}) first proved by \cite{reineke2010poisson} by 
using quiver representations. 
Also, the relations 
(\ref{A11}) and (\ref{A22}) 
were proved by cluster mutation technique
by \cite{R20}. 
\par
In \S \ref{N2NNQ}, we 
introduce dilogarithm elements and
the pentagon relations.
In \S \ref{IN1N2N3}, 
we prove a generalization of 
the formula (\ref{A11}). 
In \S \ref{c1n01n10}, 
we prove a generalization of 
the formula (\ref{A22}) 
by using a results of \S \ref{IN1N2N3}.
\section{Dilogarithm elements and pentagon relation}
\label{N2NNQ}
Let $N$ be a rank $2$ lattice with 
a skew-symmetric bilinear form
\begin{equation*}
    \{\cdot, \cdot\}
    \colon N \times N \longrightarrow 
    \mathbb{Q}.
\end{equation*}
Let $e_1$, $e_2$ be 
a basis of $N$, and we define
$$N^+:=\{a_1 e_1 + a_2 e_2 \mid a_1, a_2 \in \mathbb{Z}_{\ge 0}, a_1 + a_2 >0\}.$$
Let $\bbk$ be a field of characteristic 0, and we define an $N^+$-graded Lie algebra 
$\mathfrak{g}$ over $\bbk$ with generators $X_n$ such that
\begin{align*}
    \mathfrak{g}=\bigoplus_{n\in N^+}\mathfrak{g}_n, 
    \quad
    \mathfrak{g}_n = \bbk X_n,
    \quad
    [X_n, X_{n'}] = \{n, n'\}X_{n+n'}.
\end{align*}
Let
$\mathcal{L} := \{L \subset N^+ \mid N^+ + L \subset L, 
\#(N^+ \setminus L) < \infty$\}.
For $L \in \mathcal{L}$, we define a Lie algebra ideal 
$\mathfrak{g}^L:=\bigoplus_{n \in L} \mathfrak{g}_n$.
and the quotient of $\mathfrak{g}$ by $\mathfrak{g}^L$ 
\begin{equation*}
    \mathfrak{g}_L
    := \mathfrak{g}/\mathfrak{g}^L
    = \bigoplus_{n \in N^+ \setminus L} \mathfrak{g}_n
    \quad
    (\text{as a vector space}).
\end{equation*}
Let $G_L$ be a group with a set bijection
\begin{equation*}
    \exp_L \colon 
    \mathfrak{g}_L \longrightarrow 
    G_L
\end{equation*}
and the product is defined by a
Baker-Campbell-Hausdorff (BCH) 
formula:
\begin{equation}
    \exp_L(X)\exp_L(Y)
    = \exp_L(X + Y + \frac{1}{2}[X, Y] + \frac{1}{12} [X, [X, Y]] 
    - \frac{1}{12} [Y, [X, Y]]  + \cdots).
\end{equation}
This product formula is well-defined because $\mathfrak{g}_L$ is nilpotent.
\par
For $L, L' \in \mathcal{L}$ such that $L \subset L'$, there exists the canonical 
Lie algebra homomorphism 
$\mathfrak{g}_L \longrightarrow \mathfrak{g}_L'$, which induces 
the group homomorphism 
$G_L \longrightarrow G_L'$.
Thus, by the inverse limit we obtain 
a Lie algebra $\hat{\mathfrak{g}}$ and a group $G$: 
\begin{equation*}
    \hat{\mathfrak{g}}
    :=\varprojlim_{L \in \mathcal{L}}\mathfrak{g}_L , \quad
    G := \varprojlim_{L \in \mathcal{L}} G_L.
\end{equation*}
There is a set bijection
\begin{equation*}
    \exp \colon
    \hat{\mathfrak{g}}
     \longrightarrow G, 
    \quad
        (X_L)_{L\in \mathcal{L}} \longmapsto
    (\exp_L(X_L))_{L\in \mathcal{L}}.
\end{equation*}
We use an infinite sum to express 
an element of $\hat{\mathfrak{g}}$.
\par
We define important elements in $G$:
\begin{dfn}[Dilogarithm element]
For any $n \in N^+$, define 
\begin{equation*}
    [n] := \exp \left(\sum_{j>0} \frac{(-1)^{j+1}}{j^2} 
    X_{j n}\right) \in G.
\end{equation*}
    We call $[n]$ a \emph{dilogarithm element} for $n$.
\end{dfn}
For $c \in \mathbb{Q}$ and $g = \exp(X) \in G$, we define $g^c := \exp(cX)$.
\begin{prop}[Pentagon relation
\cite{gross2018canonical},
\cite{nakanishiIII}]
Let $n, n' \in N^+$.
Then, the following relations hold in $G$:
\begin{enumerate}
    \item If $\{n', n\} = 0$, then
    $[n'][n]=[n][n']$,
    \item If $\{n', n\} = c^{-1}$ $(c \in \mathbb{Q} \setminus {0})$, then
    \begin{equation*}
        [n']^c [n]^c = 
        [n]^c [n+n']^c [n']^c
        \quad \text{(pentagon relation)}.
    \end{equation*}
\end{enumerate}
\end{prop}
\section{Proof of formula (\ref{A11})}\label{IN1N2N3}
For a subset $I=\{i_1 < i_2 < i_3 <\cdots\}$ of $\Z$ 
and a sequence $(a_i)_{i\in I}$
of elements of $G$, 
we write 
\begin{equation*}
    \prod_{i \in I}^{\longrightarrow} a_i := 
    a_{i_1} a_{i_2} a_{i_3} \cdots
    , \quad
    \prod_{i \in I}^{\longleftarrow} a_i := 
    \cdots
    a_{i_3} a_{i_2} a_{i_1}.
\end{equation*}
For example, $\prod_{i \ge 0}
^{\longrightarrow} a_i = 
    a_{0} a_{1} a_{2} \cdots
$ and $
    \prod_{i\ge 0}
    ^{\longleftarrow} a_i = 
    \cdots
    a_2 a_{1} a_{0}$.
\par
The following is the main theorem 
of this section:
\begin{thm}[]\label{thm1}
If $\{n', n\} = c^{-1}$ $(c \in \mathbb{Q}\setminus{\{0\}})$, 
then
\begin{equation} \label{eq0}
    [n']^{2c} [n]^{2c}
    =
    \prod_{p \ge 0}
    ^{\longrightarrow}
    [n+p(n+n')]^{2c}
    \prod_{p \ge 0}
    [2^p(n+n')]^{4c/2^p}
    \prod_{p\ge 0}^{\longleftarrow}
    [n'+p(n+n')]^{2c}.
\end{equation}
\end{thm}
The case of $c=1$, $n=[(1, 0)]$, $n'=[(0, 1)]$ 
is nothing but the formula (\ref{A11}).
\par
To prove this theorem, we introduce some
notations and lemmas.
\par
Let $L \in \mathcal{L}$.
For two elements $g_1$, $g_2$ of $G$, 
let us denote $g_1 \equiv g_2 \mod L$ 
if their images in $G_L$ 
are identical. 
For example, if $n \in N^+$ is in $L$, then $[n] \equiv \exp(0) = 1_G \mod L$.
By the definition of $G$, two 
elements $g_1, g_2$ of $G$ are identical 
if and only if 
$g_1 \equiv g_2 \mod L$ 
for all $L \in \mathcal{L}$. 
\par
\begin{lem}
If $\{n',n\}=c^{-1}$ 
$(c \in \mathbb{Q}\setminus\{0\})$, 
we obtain
\begin{align}
    [n']^c [n]^{2c}=
    [n]^{2c} [2n+n']^c 
    [n+n']^{2c} [n']^{c},  \label{b2}
    \\
    [n']^{2c} [n]^c=
    [n]^c [n+n']^{2c} [n+2n']^{c} [n']^{2c}.
    \label{B2}
\end{align}
\end{lem}
\begin{proof}
The equality (\ref{b2}) can be proved by repeatedly applying 
the pentagon relation:
\begin{align*}
    [n']^{2c} [n]^c
    &= [n']^c [n']^c [n]^c
    = [n']^c [n]^c [n+n']^c [n']^c \\
    &= [n]^c [n+n']^c [n']^c [n+n']^c [n']^c 
    =[n]^c [n+n']^{2c} [n+2n']^c [n']^{2c}.
\end{align*}
Note that
$\{n', n+n'\} = c^{-1}$, thus 
we can use the pentagon identity in 
the last equality.
\par
If $\{n',n\} = c^{-1}$, then 
$\{n, n'\} = (-c)^{-1}$.
By the equality (\ref{b2}),
\begin{align}\label{b2inv}
        [n]^{-c} [n']^{2(-c)}=
    [n']^{2(-c)} [n+2n']^{-c} 
    [n+n']^{2(-c)} [n]^{-c}.
\end{align}
Taking the inverse of both sides of (\ref{b2inv}), we obtain the equality (\ref{B2}).
\end{proof}
\par
\par
The following is a key lemma:
\begin{lem} \label{lem1}
Let $l$ be a non-negative integer,
and let $n, n' \in N^+$.
If $\{n', n\} = c^{-1}$, 
\begin{equation}\label{eq6}
    [n']^{2c}
    \left(\prod_{0 \le p \le l}^{\longrightarrow}
    [n + 2pn']^c\right)
    =
    [n]^c
    \left(\prod_{1\le p \le 2l+1}
    ^{\longrightarrow}
    [n + pn']^{2c}\right)
    [n + (2l+2)n']^c [n']^{2c}
\end{equation}
\end{lem}
\begin{proof}
We will prove it by induction on $l$.
\par
If $l=0$, the equality (\ref{eq6}) is nothing but (\ref{B2}). 
\par
Let $l > 0$. 
Suppose that the claim is 
true in the case of $l-1$, then
by the induction hypothesis, 
\begin{align*}
    &[n']^{2c}
    \left(\prod_{0 \le p \le l}^{\longrightarrow}
    [n + 2pn']^c\right) \\
   & =
    [n']^{2c}
    \left(\prod_{0 \le p \le l-1}^{\longrightarrow}
    [n + 2pn']^c\right)
    [n+2ln']^c \\
    &=
    [n]^c
    \left(\prod_{1\le p \le 2l-1}
    ^{\longrightarrow}
    [n + pn']^{2c}\right)
    [n+2ln']^c [n']^{2c} [n+2ln']^c \\
    &=
    [n]^c
    \left(\prod_{1\le p \le 2l+1}
    ^{\longrightarrow}
    [n + pn']^{2c}\right)
    [n + (2l+2)n']^c [n']^{2c}.
\end{align*}
In the last equality, we use 
\begin{equation*}
    [n']^{2c} [n+2ln']^c
=[n+2ln']^c [n+(2l+1)n']^{2c} [n+(2l+2)n']^c
[n']^{2c},
\end{equation*}
which is a specialization of (\ref{B2}).
\end{proof}
Now we consider the limit of Lemma \ref{lem1}:
\begin{lem}\label{cor1}
If $\{n',n\} = c^{-1}$, then
\begin{align}
    [n']^{2c}     \left(\prod_{p \ge 0}
    ^{\longrightarrow}
    [n + 2pn']^c\right)
    =
    [n]^c 
    \left(\prod_{p \ge 1}
    ^{\longrightarrow}
    [n + pn']^{2c}\right)
    [n']^{2c} ,
    \label{eq3}
    \\
    \left(\prod_{p \ge 0}
    ^{\longleftarrow}
    [n' + 2pn]^c\right)
    [n]^{2c}
    =
    [n]^{2c} 
    \left(\prod_{p \ge 1}
    ^{\longleftarrow}
    [n' + pn]^{2c}\right)
    [n']^{c} .
    \label{eq4}
\end{align}
\end{lem}
\begin{proof}
Let $L \in \mathcal{L}$.
Then, by the cofiniteness of $L$, there exists some positive integer $l$ such that $n+2ln' \in L$.
Then, by Lemma \ref{lem1},
we obtain
\begin{align*}
    [n']^{2c} \left( \prod_{p \ge 0}^{\longrightarrow}
    [n+2pn']^c \right)
    &\equiv
    [n']^{2c} \left( 
    \prod_{0 \le p \le l}
    ^{\longrightarrow}
    [n+2pn']^c \right) \mod L \\
   & =
    [n]^{c} \left( 
    \prod_{1\le p\le 2l+1}
    ^{\longrightarrow}
    [n+pn']^{2c} \right)
    [n+(2l+2)n]^c [n']^{2c} \\
    &\equiv 
    [n]^c 
    \left(\prod_{p \ge 1}
    ^{\longrightarrow}
    [n + pn']^{2c}\right)
    [n']^{2c} \mod L.
\end{align*}
Thus, the equality (\ref{eq3}) holds.
\par
Since $\{n,n'\}=(-c)^{-1}$, 
by (\ref{eq3}), we obtain
\begin{equation} \label{eq5}
    [n]^{-2c}     \left(\prod_{p \ge 0}
    ^{\longrightarrow}
    [n' + 2pn]^{-c}\right)
    =
    [n']^{-c} 
    \left(\prod_{p \ge 1}
    ^{\longrightarrow}
    [n' + pn]^{-c}\right)
    [n]^{-2c}. 
\end{equation}
The the equality (\ref{eq4}) is
obtained by taking the inverse 
of the both sides of (\ref{eq5}).
\end{proof}
\begin{proof}[Proof of Theorem \ref{thm1}]
For $L \in \mathcal{L}$ 
and $k \in \mathbb{Z}_{>0}$, 
let $P_L(k)$ 
be
the following assertion:
for $n,n'\in N^+$ and 
$c \in \mathbb{Q}\setminus\{0\}$, if $\{n',n\}=c^{-1}$ and $k(n+n')\in L$, 
then
\begin{align}
    &[n']^{2c} [n]^{2c}\notag\\
    &\equiv
    \prod_{p \ge 0}
    ^{\longrightarrow}
    [n+p(n+n')]^{2c}
    \prod_{p \ge 0}
    [2^p(n+n')]^{4c/2^p}
    \prod_{p\ge 0}^{\longleftarrow}
    [n'+p(n+n')]^{2c}
    \mod L. \label{eq:modL}
\end{align}
For any $L\in \mathcal{L}$,
there exists some
positive integer $k$ such that
$k(n+n') \in L$.
Thus, if $P_L(k)$ is true for any
$k \in \mathbb{Z}_{>0}$ and $L \in \mathcal{L}$, then
a relation (\ref{eq:modL}) 
holds for any $L\in \mathcal{L}$, and Theorem \ref{thm1} is proved.
Fix $L \in \mathcal{L}$, and we prove $P_L(k)$ by induction on $k$.
\par
If $k=1$, the right hand side of (\ref{eq:modL}) is 
equivalent to 
$[n]^{2c} [n']^{2c}$ because 
$l n+ l' n' \in L$
for any $l, l' \in \mathbb{Z}_{\ge 1}$.
Since
\begin{align*}
    [n']^{c} [n]^{c}
    =[n]^c [n+n']^c [n']^c
    \equiv
    [n]^c [n']^c \mod L, 
\end{align*}
we obtain
$[n']^{2c} [n']^{2c} \equiv [n']^{2c} [n']^{2c}  \mod L$.
\par
Let $k \ge 2$, and we suppose
a proposition
$P_L(k-1)$ is true.
By the equality (\ref{B2}), 
\begin{align*}
    [n']^{2c} [n]^{2c}
    &= ([n']^{2c} [n]^c) [n]^c \\
    &=[n]^c [n+n']^{2c}
    [n+2n']^c ([n']^{2c} [n]^c) \\
    &= [n]^c [n+n']^{2c}
    ([n+2n']^c [n]^c) [n+n']^{2c} 
    [n+2n']^c [n']^{2c}.
\end{align*}
Since $(k-1)(n+(n+2n')) \in L$ 
and $\{n+2n', n\}=(c/2)^{-1}$, by the induction hypothesis, 
\begin{align*}
    &[n+2n']^c [n]^c \\
    &= [n+2n']^{(c/2)\cdot 2} [n]^{(c/2)\cdot 2} \\
    &\stackrel{(\ref{eq:modL})}{\equiv} \prod_{p \ge 0}
    ^{\longrightarrow}
    [n+p(2n+2n')]^{2\cdot(c/2)}
    \prod_{p \ge 0}
    [2^p(2n+2n')]^{4\cdot(c/2)/2^p} \\
    &\quad\times\prod_{p\ge 0}^{\longleftarrow}
    [(n+2n')+p(2n+2n')]^{2\cdot(c/2)}
    \quad \mod L \\
    &=\prod_{p \ge 0}
    ^{\longrightarrow}
    [n+2p(n+n')]^{c}
    \prod_{p \ge 1}
    [2^{p}(n+n')]^{4c/2^{p}}
    \prod_{p\ge 0}^{\longleftarrow}
    [(n+2n')+2p(n+n')]^{c}
    .
\end{align*}
Since $\{n+n', n\}=c^{-1}$ 
and $\{n+2n', n+n'\}=c^{-1}$,
by Lemma \ref{cor1}, 
\begin{align*}
    [n']^{2c} [n]^{2c}
    &\equiv
    [n]^c
    [n+n']^{2c}
    \prod_{p \ge 0}
    ^{\longrightarrow}
    [n+2p(n+n')]^{c}
    \prod_{p \ge 1}
    [2^{p}(n+n')]^{4c/2^{p}} \\
    &\quad\times\prod_{p\ge 0}^{\longleftarrow}
    [(n+2n')+2p(n+n')]^{c}\\
    &\quad\times
    [n+n']^{2c} 
    [n+2n']^c 
    [n']^{2c} \mod L \\
    &=[n]^c [n]^c
    \left(
    \prod_{p \ge 1}
    ^{\longrightarrow}
    [n+p(n+n')]^{c}
    \right)
    [n+n']^{2c}
    \prod_{p \ge 1}
    [2^{p}(n+n')]^{4c/2^{p}}\\
    &\quad\times
    [n+n']^{2c}
    \prod_{p\ge 1}^{\longleftarrow}
    [(n+2n')+p(n+n')]^{c}\\
    &\quad\times
    [n+2n']^{c} 
    [n+2n']^c 
    [n']^{2c} \\
    &=
    \prod_{p \ge 0}
    ^{\longrightarrow}
    [n+p(n+n')]^{2c}
    \prod_{p \ge 0}
    [2^p(n+n')]^{4c/2^p}
    \prod_{p\ge 0}^{\longleftarrow}
    [n'+p(n+n')]^{2c}
\end{align*}
This completes the proof of Theorem \ref{thm1}.
\end{proof}
\section{Proof of formula (\ref{A22})}
\label{c1n01n10}
The formula (\ref{A22}) is the case of 
$c=-1$, $n=[(0,1)]$, $n'=[(1,0]$ of the following theorem:
\begin{thm}\label{thm2}
If $\{n', n\}=c^{-1}$ $(c \in \Q)$, 
then we obtain
\begin{align}
    [n']^c [n]^{4c}
    &
    =
    \prod_{p\ge 0}^{\longrightarrow}
    ([(2p+1)n+p n']^{4c} 
    [(4p+4)n+(2p+1)n']^c) \notag
    \\
    &\quad\times 
    [2n+n']^{2c} \prod_{p\ge 0} 
    [2^p (2n+n')]^{4c / 2^p} \notag
    \\
    &\quad\times
    \prod_{p\ge 0}^{\longleftarrow}
    ([(2p+1)n+(p+1)n']^{4c}
    [4pn+(2p+1)n']^c). \notag
\end{align}
\end{thm}
To prove this theorem, 
we consider some lemmas.
\begin{lem}\label{lem:1414}
If
$\{n',n\}=c^{-1}$ $(c \in \mathbb{Q}\setminus\{0\})$, then we obtain
\begin{align}
    &[n']^c 
    \left(
    \prod_{0\le p\le l}^{\longrightarrow}[n+pn']^{2c}
    \right) \notag\\
    &=
    [n]^{2c}
    \left(
    \prod_{1\le p\le l}^{\longrightarrow}
    ([2n+(2p-1)n']^c 
    [n+pn']^{4c}) \notag
    \right) \\
    &\quad\times
    [2n+(2l+1)n']^c 
    [n+(l+1)n']^{2c}
    [n']^c
    \notag
\end{align}
\end{lem}
\begin{proof}
We prove it by induction on $l$.
The case of $l=0$ is nothing less than
the equality (\ref{b2}).
\par
Let $l > 0$. 
Suppose that the claim is 
true in the case of $l-1$, then
\begin{align*}
    &[n']^{c}
    \left(\prod_{0 \le p \le l}^{\longrightarrow}
    [n + pn']^{2c}\right) \\
   & =
    [n']^{c}
    \left(\prod_{0 \le p \le l-1}^{\longrightarrow}
    [n + pn']^{2c}\right)
    [n+ln']^{2c} \\
    &=
   [n]^{2c}
    \left(
    \prod_{1\le p\le l-1}^{\longrightarrow}
    ([2n+(2p-1)n']^c 
    [n+pn']^{4c}) \notag
    \right) \\
    &\quad\times
    [2n+(2l-1)n']^c 
    [n+ln']^{2c}
    [n']^c [n+ln']^{2c}\\
    &=
    [n]^{2c}
    \left(
    \prod_{1\le p\le l}^{\longrightarrow}
    ([2n+(2p-1)n']^c 
    [n+pn']^{4c}) \notag
    \right) \\
    &\quad\times
    [2n+(2l+1)n']^c 
    [n+(l+1)n']^{2c}
    [n']^c
    \end{align*}
In the last equality, we use 
\begin{equation*}
    [n']^{c} [n+ln']^{2c}
=[n+ln']^{2c} [2n+(2l+1)n']^{c} [n+(l+1)n']^{2c}
[n']^{2},
\end{equation*}
which is a specialization of (\ref{b2}).
\end{proof}
Now we consider the limit of Lemma \ref{lem:1414}:
\begin{lem}\label{limitoflem4}
If $\{n',n\} = c^{-1}$, then we obtain
\begin{align}
    [n']^{c}     
    \left(\prod_{p \ge 0}
    ^{\longrightarrow}
    [n + pn']^{2c}\right)
    =
    [n]^{2c} 
    \left(\prod_{p \ge 1}
    ^{\longrightarrow}
    [2n + (2p-1)n']^c 
    [n+pn']^{4c}\right)
    [n']^{c} ,
    \label{eq:11}
    \\
    \left(\prod_{p \ge 0}
    ^{\longleftarrow}
    [n' + pn]^{2c}\right)
    [n]^{c}
    =
    [n]^{c} 
    \left(\prod_{p \ge 1}
    ^{\longleftarrow}
    [n'+pn]^{4c}
    [2n' + (2p-1)n]^c 
    \right)
    [n']^{2c} .
    \label{eq:12}
\end{align}
\end{lem}
\begin{proof}
Let $L \in \mathcal{L}$.
Then, there exist 
some positive 
integer $l$ such that
$n+ln'\in L$.
Then, by Lemma \ref{lem:1414}, 
we obtain
\begin{align*}
    &[n']^c
    \prod_{p\ge 0}^{\longrightarrow}
    [n+pn']^{2c}\\
    \equiv 
    &[n']^c 
    \left(
    \prod_{0\le p\le l}^{\longrightarrow}[n+pn']^{2c}
    \right) \mod L\\
    &=
    [n]^{2c}
    \left(
    \prod_{1\le p\le l}^{\longrightarrow}
    ([2n+(2p-1)n']^c 
    [n+pn']^{4c}) 
    \right) \\
    &\quad\times
    [2n+(2l+1)n']^c 
    [n+(l+1)n']^{2c}
    [n']^c
    \\
    &\equiv
    [n]^{2c} 
    \left(\prod_{p \ge 1}
    ^{\longrightarrow}
    [2n + (2p-1)n']^c 
    [n+pn']^{4c}\right)
    [n']^{c} \mod L
\end{align*}
Thus, the equality (\ref{eq:11})
holds.
\par
Since $\{n, n'\}=(-c)^{-1}$, by 
(\ref{eq:11}), we obtain
\begin{align}
        [n]^{-c}     
    \left(\prod_{p \ge 0}
    ^{\longrightarrow}
    [n' + pn]^{-2c}\right)
    =
    [n']^{-2c} 
    \left(\prod_{p \ge 1}
    ^{\longrightarrow}
    [2n' + (2p-1)n]^{-c} 
    [n'+pn]^{-4c}\right)
    [n]^{-c}.\notag
\end{align}
By taking the inverse of both sides of 
this equality, we have the equality
(\ref{eq:12}).
\end{proof}
\begin{proof}[Proof of Theorem \ref{thm2}]
Using the pentagon relations, 
Theorem \ref{thm1} and
Lemma \ref{limitoflem4}, we can calculate as follows:
\begin{align*}
    &[n']^c [n]^{4c} 
    \\
    &= [n']^c [n]^{2c} [n]^{2c}
    \\
    &= [n]^{2c} [2n+n']^c 
    [n+n']^{2c} [n']^c [n]^{2c}
    \\
    &= [n]^{2c} [2n+n']^c [n+n']^{2c} 
    [n]^{2c} [2n+n']^c [n+n']^{2c} [n]^c
    \\
    &\buildrel{(\ref{eq0})}\over{=} [n]^{2c} [2n+n']^c 
    \\
    &\quad
    \times \prod_{p\ge 0}^{\longrightarrow} [n+p(2n+n')]^{2c} 
    \\
    &\quad
    \times\prod_{p\ge 0} [2^p(2n+n')]^{4c/2^p}
    \\
    &\quad
    \times\prod_{p\ge 0}^{\longleftarrow}
    [n+n'+p(2n+n')]^{2c}
    \\
    &\quad\times[2n+n']^c[n+n']^{2c}[n']^c
    \\
    &\buildrel{(\ref{eq:11}), (\ref{eq:12})}\over{=}
    [n]^{2c} \times [n]^{2c} 
    \left(\prod_{p\ge 1}^{\longrightarrow}
    ([2n+(2p-1)(2n+n')]^c
    [n+p(2n+n')]^{4c})\right)
    [2n+n']^c
    \\
    &\quad\times
    \prod_{p\ge 0}[2^p(2n+n')]^{4c/2^p}
    \times
    [2n+n']^c
    \\&\quad
    \times\left(\prod_{p\ge 1}^{\longleftarrow}
    ([(n+n')+p(2n+n')]^{4c}
    [2(n+n')+(2p-1)(2n+n')]^c)\right)
    \\&\quad
    \times[n+n']^{2c}[n+n']^{2c}[n']^c
    \\&
    =
    \prod_{p\ge 0}^{\longrightarrow}
    ([n+p(2n+n')]^{4c}[2n+(2p+1)(2n+n')]^c)
    \\&\quad
    \times [2n+n']^{2c} \prod_{p\ge 0} 
    [2^p(2n+n')]^{4c/2^p}
    \\&\quad
    \times
    \prod_{p\ge 0}^{\longleftarrow}
    [(n+n')+p(2n+n')]^{4c}
    [2(n+n')+(2p-1)(2n+n')]^c 
    \\&
    =
    \prod_{p\ge 0}^{\longrightarrow}
    ([(2p+1)n+p n']^{4c} 
    [(4p+4)n+(2p+1)n']^c) \notag
    \\
    &\quad\times 
    [2n+n']^{2c} \prod_{p\ge 0} 
    [2^p (2n+n')]^{4c / 2^p} \notag
    \\
    &\quad\times
    \prod_{p\ge 0}^{\longleftarrow}
    ([(2p+1)n+(p+1)n']^{4c}
    [4pn+(2p+1)n']^c
\end{align*}
This completes the proof.
\end{proof}
\bibliography{affine}
\bibliographystyle{alpha}
\end{document}